\documentclass{amsart}
\usepackage{cleveref}
\usepackage{amssymb}
\usepackage{fullpage}
\usepackage[colorinlistoftodos,bordercolor=orange,backgroundcolor=orange!20,linecolor=orange,textsize=scriptsize]{todonotes}

 \renewcommand{\ge}{\geqslant}
 \renewcommand{\le}{\leqslant}

  \newcommand{\y}{\mathbf{y}}
  \newcommand{\x}{\mathbf{x}}
  \newcommand{\z}{\mathbf{z}}
  
  \newcommand{\0}{\mathbf{0}}

 \newcommand\bk{{\mathbf k}}

 \newcommand\bu{\mathbf{u}}
 \newcommand\bv{\mathbf{v}}
 \newcommand\bw{{\mathbf w}}

 \newcommand\cA{{\mathcal A}}
 \newcommand\cB{{\mathcal B}}

 \newcommand\cH{{\mathcal H}}

 \newcommand\cK{{\mathcal K}}
 \newcommand\cL{{\mathcal L}}
 \newcommand\cM{{\mathcal M}}
 \newcommand\cN{{\mathcal N}}

 \newcommand\cT{{\mathcal T}}

 \newcommand\rH{{\rm H}}

  \newcommand{\e}{\mathbf{e}}
  \newcommand{\f}{\mathbf{f}}
  \newcommand{\g}{\mathbf{g}}
  \newcommand{\h}{\mathbf{h}}

  \newcommand{\trans}{^\top}

\newcommand{\tr}{\operatorname{tr}}
\newcommand{\C}{\mathbb{C}}

\newcommand{\N}{\mathbb{N}}

\newtheorem{theorem}{Theorem}

\newtheorem{lemma}[theorem]{Lemma}
\newtheorem{corollary}[theorem]{Corollary}
\theoremstyle{definition}

\theoremstyle{remark}

\begin{document}
\title{Infinite dimensional generalizations of Choi's theorem}
\author{Shmuel Friedland\\
Department of Mathematics, Statistics and Computer Science,\\
 University of Illinois at Chicago, Chicago, Illinois 60607-7045, USA }
  \date{July 6, 2018 }
\maketitle
\begin{abstract}
In this paper we give a simple sequence of necessary and sufficient finite dimensional conditions for a positive map between certain subspaces of bounded linear operators on separable Hilbert spaces to be completely positive.  These criterions are natural generalization of Choi's  characterization for completely positive maps between pairs of linear operators on finite dimensional Hilbert spaces.  
We apply our conditions to a completely positive map  between two trace class operators on separable Hilbert spaces.  A map $\mu$ is called a quantum channel, if it is trace preserving, and $\mu$ is called a quantum subchannel if it decreases the trace of a positive operator.  We give simple neccesary and sufficient condtions for $\mu$ to be a quantum subchannel.   
We show that $\mu$ is a quantum subchannel if and only if it has Hellwig-Kraus representation.  The last result extends the classical results of Kraus and the recent result of Holevo for characterization of a quantum channel.
\end{abstract}

 \section{Introduction}\label{sec:intro}
 A quantum system in von Neumann model is described by a compact positive operator acting on a separable Hilbert $\cH$ space with trace one.  We denote by $T_1(\cH)$ the closed Banach space of trace class operators on $\cH$.  The evolution of an open quantum system from one system to another one can be described by a  bounded positive linear operator $\mu$ between a pair of trace class operators of two separable Hilbert spaces $\cH_1$ and $\cH_2$: $\mu:T_1(\cH_1)\to T_1(\cH_2)$.  The standard assumption that $\mu$ is trace preserving.  A more general assumption that $\mu$ decrease the trace of a positive operator \cite{HK69}. It was realized by Kraus \cite{Kra71} that physical assumptions yield that $\mu$ has to be completely positive.  A completely positive trace preserving $\mu$ is called a quantum channel.  We call a completely positive $\mu$ which decreases the trace of positive operators a quantum subchannel.  The purpose of this paper is threefold. First, we give a simple sequence of necessary and sufficient finite dimensional conditions for a positive map between certain subspaces of bounded linear operators on separable Hilbert spaces to be completely positive. These criterions are natural generalization of Choi's  characterization for completely positive maps between pairs of linear operators on finite dimensional Hilbert spaces.  Second,  we give a simple sequence of necessary and sufficient finite dimensional conditions for  that $\mu$ is a quantum subchannel. Third, we show that any quantum subchannel has Hellwig-Kraus representation. The last result extends the classical result of Kraus \cite{Kra71} and the recent result of Holevo \cite{Hol11} for characterization of a quantum channel.  
 
We now explain in technical details the main results of this paper.  The notion of a completely positive operator between two $C^*$-algebra was introduced by Stinespring \cite{Sti55}.  In his seminal paper Stinespring gave necessary and sufficient conditions for a positive map $\mu$ between two $C^*$ algebras to be completely positive.  This is the celebrated Stinespring's dilation (representation) theorem.  For the finite dimensional case, that is, the two $C^*$ algebras are the the spaces of square complex valued matrices, $\mu$ has a simple representation.  However, given a linear transformation $\mu$ from one space of square complex valued matrices to another space of square complex matrices, it is not simple to decide if $\mu$ is completely positive.  The simple necessary and sufficient conditions for $\mu$ to be completely positive were given by Choi in his fundamental paper \cite{Cho75}.  In \cite{HK69} the authors considered the evolution of an open quantum system.  They showed that under corresponding quantum physics assumption the map $\mu$ on the trace class to itself has a special representation \cite[(6)-(7)]{HK69}, which we call Hellwig-Kraus representation.  (See \eqref{HKqsc} for the case $\cK=\cH$.)  Kraus proved in \cite[Theorem 3.3]{Kra71} that a quantum subchannel has representation \eqref{HKqsc} for $\cK=\cH$ assuming that $\mu$ is continuous with respect to ultraweak (weak operator) topology. 
Generalizations to quantum channels on a pair trace class operators were given by Holevo \cite{Hol11} and by St{\o}rmer to a von Neumann factor with a separating and cyclic vector \cite{Sto15}.

In the first part of this paper we give a natural generalization of Choi's theorem for a bounded linear $\mu: \Sigma_1\to \Sigma_2$, where $\Sigma_i$
is given as a subspace of bounded linear operators $B(\cH_i)$, endowed with the norm $\|\cdot\|_i$.  (We assume that there exists a constant $K_i>0$ such that $\|g\|\le K_i\|g\|_i$, where $\|g\|$ is the operator norm on $B(\cH_i)$.)  We impose the following conditions on $\Sigma_i$.  
The first condition is that each $\Sigma_i$ is invariant under the conjugation: $\Sigma_i^*=\Sigma_i$.  The second condition is that each $\Sigma_i$ contains the subspace of all  finite range operators.
These two condition holds for all Schatten $p$-classes $T_p(\cH_i)$ of compact operators for $p\in[1,\infty]$, where $T_ \infty(\cH_i)$ is the subspace of compact operators in $B(\cH_i)$ with the standard operator norm.  The third condition is more technical and is stated later.  It is always satisfied if $\Sigma_1=T_p(\cH)$ for $p\in[1,\infty]$.   Our main object is to give a sequence of necessary and sufficient conditions,  where each condition is exactly the Choi conditions for the finite dimensional case. 

Our first main result can be stated as follows.  Assume that $\cH$ is a Hilbert space with an inner product $\langle\cdot,\cdot\rangle$.  
For $\bv\in\cH$ and $a\in B(\cH)$ we denote by $\bv^\vee$ and  $a^*\in B(\cH)$ the bounded linear functional  $\bv^ \vee(\x)=\langle \x,\bv\rangle$, and the adjoint operator   $\langle a\x,\y\rangle=\langle\x,a^*\y\rangle$ on $\cH$ respectively.
  For $\bu,\bv\in\cH$ denote by $\bu\bv^\vee\in B(\cH)$ the rank one operator $\bu\bv^\vee(\x)=\langle\x,\bv\rangle\bu$.  
Recall that a sequence $a_n\in B(\cH), n\in\N$ converges in the weak operator topology to $a\in B(\cH)$ if $\lim_{n\to\infty}\langle a_n\x,\y\rangle=\langle a\x,\y\rangle$ for all $\x,\y\in \cH$.  
Denote by $B(\cH)^{(n)}$ the space of $n\times n$ matrices $G=[g_{ij}]_{i=j=1}^n$, where $g_{ij}\in B(\cH)$ for $i,j\in[n]=\{1,\ldots,n\}$.  Then $G^*$ is the adjoint matrix: $(G^*)_{ij}=g^*_{ji}$.  The matrix $G$ is selfadjoint if $G=G^*$.  A matrix $G$ is called positive, which is denoted by $G\ge 0$, if $G$ is selfadjoint and $\sum_{i=j=1}^n \langle g_{ij}\x_i,\x_j\rangle\ge 0$ for all $\x_1,\ldots,\x_n\in\cH$.
Assume that $\mu:\Sigma_1\to\Sigma_2$ is a bounded linear map.   Then $\mu$ is called completely positive if  the following two conditions are satisfied:  First, $\mu$ commutes with the adjoint map, i.e.,  $\mu(g)^*=\mu(g^*)$ for all $g\in\cH_1$. (This is equivalent to the condition that $\mu$ preserves selfadjoint operators.)  Second, for each $n\in \N$ and $G=[g_{ij}]_{i=j=1}^n\ge 0$ the matrix $\mu(G)=[\mu(g_{ij})]_{i=j=1}^n$ is positive.

Our first generalization of Choi's result can be stated as follows:
\begin{theorem}\label{mainthminf2}  For $i=1,2$ assume that the following conditions hold:
\begin{enumerate}
\item $\cH_i$ is a Hilbert space with an orthonormal basis $\{\e_{j,i}\}, j\in\N$.  
\item For each $n\in\N$ denote by $P_{n,i}$ the projection on $\textrm{span}(\e_{1,i},\ldots,\e_{n,i})$.
\item  $\Sigma_i$ is a  subspace of $B(\cH_i)$ closed with respect to the norm $\|\cdot\|_i$ and the conjugation $^*$.  Furthermore $\Sigma_i$ contains  the subspace of all  finite range operators.
\item For each $g\in \Sigma_1$ the sequence $\mu(P_{n,1}g P_{n,1}), n\in\N$ converges in the weak operator topology to $\mu(g)$.
\end{enumerate}
Then a bounded linear operator $\mu: \Sigma_1\to \Sigma_2$ is completely positive if 
and only if for each $n\in\N$
the  matrix $L_n(\mu):=[P_{n,2} \mu(\e_{i,1}\e_{j,1}^\vee)P_{n,2}]_{i=j=1}^n\in B(\Sigma_2)^{(n)}$ is positive.
\end{theorem}

We will show that if $\Sigma_1=T_p(\cH_1)$ then the condition (4) is satisfied.
In other cases the assumption (4) is satisfied if $\mu$ maps any convergent sequence in the strong operator topology to a convergent sequence in the weak operator  topology.  A stronger condition than (4) is that $\mu$ is normal.  That is, if $a_n\in\Sigma_1,n\in\N$ is a sequence of increasing positive operators  that converges
to $a\in \Sigma_1$ in the weak operator topology, then $\mu(a_n),n\in\N$ is an increasing sequence that converges to $\mu(a)$ in the weak operator topology.

Using a standard model of an interaction of an open quantum system  with another quantum system Hellwig and Kraus \cite[(6)-(7)]{HK69} showed that $\mu$ has the following form:  $\mu(a)=\sum_{i=1}^{\infty} A_i a A_i^\vee$.  Here we assume that $A_i\in L(\cH_1,\cH_2)$, $A_i^\vee:\cH_2^\vee\to \cH_1^\vee$, (the dual operator), for $i\in\N$.  Furthermore we have the inequality $\sum_{i=1}^\infty A_i^\vee A_i\le Id$.
To be precise it is assumed in \cite{HK69} that $\cH_1=\cH_2$.  We call such $\mu$ a quantum subchannel.  It is straightforward to show that a quantum subchannel is a quantum channel if and only iff $\sum_{i=1}^\infty A_i^\vee A_i=Id$.  
It was shown by Kraus \cite[Theorem 3.3]{Kra71} that assuming that $\cH_1=\cH_2$ and $\mu$ continuous with respect the weak operator topology a quantum subchannel has the above representation).  Our second main result is that a subchannel has the above Hellwig-Kraus representation (without the assumption that $\mu$ is continuous with respect to the weak operator topology.
For a quantum channel this result was shown by Holevo \cite{Hol11}. (In this case 
$\sum_{i=1}^\infty A_i^\vee A_i=Id$, and the above representation is called Kraus representation.)  

We now survey briefly the contents of this paper.  In \S\ref{sec:SCthms} we recall the Stinespring dilatation theorem \cite{Sti55} and Choi's theorem on characterization of completely positive operators between two spaces of complex square matrices \cite{Cho75}.  In \S\ref{sec:prfmthm} we state Theorem \ref{mainthm}, which is a variation of Theorem \ref{mainthminf2}.  We then prove the two theorems.  In \S\ref{sec:addres} we show that for $\Sigma_1=T_p(\cH_1)$ the condition (4) of Theorem \ref{mainthminf2} is satisfied. In \S\ref{sec:repqc} we discuss notions  quantum channels and subchannels on finite and infinite dimensional subspaces.  
We show that any quantum subchannel has Hellwig-Kraus representation.  
\section{Stinespring and Choi theorems}\label{sec:SCthms}
 Let $\cA$  be a $C^*$-algebra and denote by $K(\cA)$ the cone of positive operators.  So $0\le a\iff a\in K(\cA)\iff a=bb^*$.  
Assume that $\cB$ is another $C^*$ algebra.  A bounded linear map $\mu: \cA\to\cB$ is called positive, and denote by $\mu\ge 0$, if $\mu(a)\ge 0$ for each $a\ge 0$.  A completely positive operator $\mu$ was defined by Stinespring \cite{Sti55} as follows:  Let $\cA^{(n)}$ be the algebra
of $n\times n$ matrices $G=[g_{ij}]_{i=j=1}^n$,where $g_{ij}\in \cA$ for $i,j=[n]=\{1,\ldots,n\}$.  Then $\cA^{(n)}$ is a $C^*$ algebra, where $(G^*)_{ij}=g^*_{ji}, i,j\in[n]$.
Clearly, $\mu$ induces the linear bounded maps $\mu_n:\cA^{(n)}\to \cB^{(n)}$ for each $n\in\N$.  The map $\mu$ is completely positive if $\mu_n\ge 0$ for each $n\in\N$. 

Let $\cK$ be a Hilbert space and denote by $B(\cK)$ the $C^*$ algebra of bounded linear operators $g:\cK\to\cK$.  Assume that $\cA$ is a $C^*$ algebra with the unit $1$.  The $\rho: C^*\to B(\cK)$ is called a $^*$-representation if $\rho$ is positive and $\rho(1)=Id$.  

Assume $\cA$ is a unital $C^*$ algebra, $\cH$ is a Hilbert space and  $\mu: \cA\to B(\cH)$ is a bounded linear operator. The Stinespring dilation theorem says that there exists a $^*$-representation  $\rho:\cA\to B(\cK)$ such that $\mu=V^*\rho V$ for a some bounded operator $V:\cH\to\cK$ if and only $\mu$ is completely positive. 

One can drop the assumption that $\cA$ is unital as follows.  Suppose that $\cA$ is not unital and $\mu:\cA\to B(\cH)$ is completely positive. Then $\cA$ is contained in a unital $C^*$-algebra $\cA_1=\cA\oplus \C$  with naturally defined operations \cite[Proposition I.1.3]{Dav96}.  Furthermore, $\mu$ lifts to  a completely positive $\mu_1:\cA_1\to B(\cH)$, where $\mu_1((a,z))=\mu(a)$.   
We call $\rho:\cA\to B(\cK)$ a representation, if $\rho$ is the restriction of a representation $\rho_1:\cA_1\to B(\cK)$.
We thus obtain the following theorem, which is usually called Stinespring's  representation theorem \cite[Theorem 6.1.1]{Ror01}:
\begin{theorem}\label{Stinespring} Let $\mu$ be a completely positive map from a $C^*$-algebra $\cA$ into $B(\cH)$. Then there is a representation $\rho$ of $\cA$ on some Hilbert space $\cK$ and a bounded linear
operator $V:\cH\to\cK$ such that $\mu(a)=V^*\rho(a) V$.
If $\cA$ has a unit $1$ and $\mu(1)=Id$ then $V$ is an isometry.
\end{theorem}

Assume that $\cH$ and $\cK$ are separable Hilbert space, i.e. $\cH$ and $\cK$ have countable orthonormal bases.  Let $\mu:B(\cK)\to B(\cH)$ be a bounded linear map. 
What are simple necessary and sufficient conditions for $\mu$ to be completely positive?  

When $\cK$ and $\cH$ are finite dimensional, the elegant necessary and sufficient conditions were given  by Choi \cite{Cho75}:  Choose orthonormal bases $\bk_1,\ldots,\bk_n$ and $\h_1,\ldots,\h_m$ in $\cK$ and $\cH$ respectively.
This will give an isomorphism $\iota_{\cK}: B(\cK)\to \C^{n\times n}, \iota_{\cH}: B(\cH)\to \C^{m\times m}$.  More precisely, for each $\0\ne\bu, \bv\in\cK$ denote by $\bu\bv^\vee$ the following rank one operator in $B(\cK)$: $(\bu\bv^\vee)(\bw)=\langle\bw,\bv\rangle\bu$, where $\langle\cdot,\cdot\rangle$ is the inner product in $\cK$.  So $\bk_i\bk_j^\vee, i,j\in[n]$ is the standard basis of $B(\cK)$ associated with the orthonormal basis of $\bk_1,\ldots,\bk_n$, which corresponds to the standard basis on $n\times n$ matrices $\C^{n\times n}$.
Then $\mu$ is completely positive if and only if the $mn$ hermitian matrix, written in a block matrix form,  is positive definite:
\begin{equation}\label{Choicond}
L(\mu)=[\mu(\bk_i\bk_j^\vee)]_{i=j=1}^n\ge 0.
\end{equation}
 \section{Proof of the main theorem}\label{sec:prfmthm}
 Let $\cH$ be a Hilbert space and assume $\{g_n\}_{ n\in\N}\subset B(\cH)$.  Then $\lim_{n\to \infty} g_n =g$: 
 \begin{enumerate}
\item in the strong operator topology, abbreviated as strongly convergent sequence, if $\lim_{n\to\infty} g_n\x=g\x$ for each $\x\in\cH$; 
\item in the weak operator topology, abbreviated as weakly convergent sequence, if 
 $\lim_{n\to\infty} \langle g_n\x,\y\rangle=\langle g\x,\y\rangle$ for each $\x,\y\in\cH$.
 \end{enumerate}
 
 Assume that $\cK$ is another Hilbert space, and let $\mu: B(\cK)\to B(\cH)$ be a bounded linear map.  We say that $\mu$ maps strongly convergent sequences to weakly convergent sequences if for each strongly convergent sequence $\lim_{n\to\infty}g_n=g$ we have that $\lim_{n\to\infty} \mu(g_n)=\mu(g)$ in the weak operator topology. 
 
 Recall  the following well known result, whose proof we leave to the reader..
\begin{lemma}\label{wconvT}  Let $\cH$ be a Hilbert space with an infinite countable basis $\{\h_n\}, n\in\N$.  Denote by $P_n:\cH\to\cH$ the orthogonal projection on span$(\h_1,\ldots,\h_n)$ for $n\in\N$.  Then for each $\x\in\cH$ and $g\in B(\cH)$ 
\[\lim_{n\to\infty} \|P_n\x-\x\|=\lim_{n\to\infty}\|gP_n\x-g\x\| =\lim_{n\to\infty}\|P_ngP_n\x-g\x\| =0.\]
\end{lemma}
The above lemma is equivalent to the statement:
\[\lim_{n\to\infty} P_n=Id,\quad \lim_{n\to\infty}gP_n=\lim_{n\to\infty} P_ngP_n=g\]
in the strong operator topology.
  
 To prove Theorem \ref{mainthminf2} it will be convenient to state a variation of Theorem \ref{mainthminf2}:
 \begin{theorem}\label{mainthm}  Let $\cK$ and $\cH$ be Hilbert spaces with countable bases $\bk_1,\ldots,\h_1,\ldots,$ respectively.  For each $m\in\N$ denote by $P_m$ the orthogonal projection of $\cH$ on span$(\h_1,\ldots,\h_m)$ if $\dim \cH\ge m$.  Assume that  $\mu: B(\cK)\to B(\cH)$ is a bounded linear operator which maps strongly convergent sequences to weak convergent sequences.  Then in each of the following cases $\mu$ is completely positive if and only if the corresponding  conditions hold:
\begin{enumerate}
\item Assume that $\cK$ has dimension $n$ with an orthonormal basis $\bk_1,\ldots,\bk_n$ and $\cH$ is infinite dimensional. Then for each $m\in\N$
the hermitian matrix $L_{n,m}(\mu):=[P_m \mu(\bk_i\bk_j^\vee)P_m]_{i=j=1}^n\in B(\cH)^{(n)}$ is positive.
\item  Assume that $\cH$ has dimension $m$ with an orthonormal basis $\h_1,\ldots,\h_m$ and $\cK$ is infinite dimensional.  Then for each $n\in\N$ the
the  matrix $L_{n,m}(\mu):=[ \mu(\bk_i\bk_j^\vee)]_{i=j=1}^n\in B(\cH)^{(n)}$ is positive.
\item Assume that $\cK$ and $\cH$ are infinite dimensional.
Then for each $n\in\N$
the matrix $L_n(\mu):=[P_n \mu(\bk_i\bk_j^\vee)P_n]_{i=j=1}^n\in B(\cH)^{(n)}$ is positive.
\end{enumerate}
\end{theorem}
Note that if $\cK$ is finite dimensional then the norm operator topology, the strong operator topology and the weak are operator topology on $B(\cK)$ are same topologies.
\begin{proof}
Assume that $\mu$ is completely positive.  Suppose that $\dim \cK\ge n$.  Let $G=[\bk_i\bk_j^\vee]_{i,j=1}^n$.  As $G=(\frac{1}{\sqrt{n}}G) (\frac{1}{\sqrt{n}}G)^*$ it follows that $G$ is positive.  Hence $\mu_n(G)=[\mu(\bk_i\bk_j^\vee)]_{i,j=1}^n$ is positive.  In particular,
\[\sum_{i=j=1}^n \langle\mu(\bk_i\bk_j^\vee)\x_j,\x_i\rangle\ge 0\]
for each $\x_1,\ldots,\x_n\in\cH$.  Assume that $\dim\cH\ge m$. Then $P_m\x_1,\ldots,P_m\x_n\in$span$(\h_1,\ldots,\h_m)$.    Therefore $L_{n,m}(\mu)$, as defined in part \emph{(1)} is positive. In particular, if $\dim\cH\ge n$ we obtain that $L_n(\mu)=L_{n,n}(\mu)$ is positive.

We now need to show the sufficiency of the conditions in (1), (2) and (3). 

 \noindent 
 (1)  We first show that $M_n(\mu):=[\mu(\bk_i\bk_j^\vee)]_{i=j=1}^n\in B(\cH)^{(n)}$ is positive.  Fix $\x_1,\ldots,\x_n\in\cH$.  As $L_{n,m}(\mu)\ge 0$ it follows that
 \[\sum_{p=q=1}^n \langle P_m\mu(\bk_p\bk_q^\vee)P_m\x_q,\x_p\rangle\ge 0.\]
 Let  $m\to\infty$ and use Lemma \ref{wconvT} to deduce that $\sum_{p=q=1}^n \langle \mu(\bk_p\bk_q^\vee)\x_q,\x_p\rangle\ge 0$.  Hence $M_n(\mu)\ge 0$. 
 
 We now show that the image of $B(\cK)^{(k)}$ by $\mu_k$ is positive for each $k\in\N$.  Recall that $B(\cK)$ is isomorphic to $\C^{n\times n}$.  Hence $B(\cK)^{(k)}$ is isomorphic to  $(\C^{n\times n})^{(k)}=\C^{(kn)\times(kn)}$.  
Denote by $\tilde \mu_k: \C^{(kn)\times(kn)}\to B(\cH)^{(k)}$ the induced map by $\mu_k$.  To show that $\tilde\mu_k$ is positive, it is enough to show that the image of rank one $(kn)\times (kn)$ matrices positive matrix by $\tilde\mu_k$ is positive.
A rank one positive matrix of order $kn$ in the block form is $U=[\bu_i\bu_j^*]_{i=j=1}^k$, where  $\bu_1=(u_{1,1},\ldots,u_{n,1})^\top,\ldots,\bu_k=(u_{1,k},\ldots,u_{n,k})^\top\in\C^n$.  This rank one matrix is represented by $V=[\bv_i \bv^\vee_j]_{i=j=1}^k\in B(\cK)^{(k)}$, where $\bv_i=\sum_{p=1}^n u_{p,i}\bk_p$ for $i\in[k]$.

A straightforward computation yields:
\[F=[F_{ij}]_{i=j=1}^{k}=\mu_k([\bv_i\bv_j^\vee]_{i=j=1}^k)=[\mu(\bv_i\bv_j^\vee)]_{i=j=1}^k=[\bu_i^\top M_{n}(\mu)\bar\bu_j]_{i=j=1}^k\in B(\cH)^{(k)}\]
is selfadjoint.  Assume that $\x_1,\ldots,\x_k\in\cH$.  Then
\begin{eqnarray*}
\sum_{i=j=1}^k \langle F_{ij}\x_j,\x_i\rangle=\sum_{p=q=1}^n\langle \mu(\bk_p\bk_q^\vee)\y_q,\y_p\rangle,\quad \y_q=\sum_{j=1}^k \bar u_{q,j}\x_j, q\in[n].
\end{eqnarray*}
Hence $F$ is positive.

\noindent
(ii)  We need to show that $\mu_k$ is positive for each $k\in\N$.  
Assume that $G=[g_{ij}]^k_{i=j=1} \in B(\cK)^{(k)}$ is positive.   Denote by $Q_n$ the orthogonal projection of $\cK$ on span$(\bk_1,\ldots,\bk_n)$.  Let $\cK_n=Q_n\cK$.  Define
$G_n=[Q_ng_{ij}Q_n]^k_{i=j=1}$.  Then $G_n\ge 0$.  Observe that $G_n\in B(\cK_n)^{(k)}$.  The assumption that $L_{n,m}(\mu):=[ \mu(\bk_i\bk_j^\vee)]_{i=j=1}^n\in B(\cH)^{(n)}$ is positive and Choi's theorem yield that $\mu_k(G_n)\ge 0$.  By Lemma \ref{wconvT} $Q_n g Q_n, n\in \N$ converges strongly to $g$ for each $g\in B(\cK)$.  As $\mu(Q_ngQ_n), n\in\N$ converges weakly to $\mu(g)$ it follows that for fixed $\x_1,\ldots,\x_k\in\cK$ we have the equaity

\[\sum_{i=j=1}^k \langle\mu(g_{ij})\x_j,\x_i\rangle=\lim_{n\to\infty}\sum_{i=j=1}^k \langle\mu( Q_ng_{ij}Q_n)\x_j,\x_i\rangle\ge 0.\]
Hence $\mu_k(G)$ is positive.

\noindent
(iii) Assume that for each $n\in\N$
the matrix $L_n(\mu):=[P_n \mu(\bk_i\bk_j^\vee)P_n]_{i=j=1}^n\in B(\cH)^{(n)}$ is positive.  Fix $m\in\N$ and let $\cH_m=P_m\cH_m$.  It is straightforward to show that for each $n\ge m$ the matrix $L_{n,m}(\mu)\in B(\cH_m)^{(n)}$, as defined in (1), is positive.  Let $\hat\mu_m: B(\cK)\to B(\cH_m)$ be given by $\hat\mu_m(g)=P_m \mu(g) P_m$.  Part (2) yields that $\hat \mu_m$ is completely positive.  That is for each $k\in\N$ and positive $G=[g_{ij}]_{i=j=1}^k\in B(\cK)^{(k)}$ the matrix $[\hat\mu(g_{ij})]_{i=j=1}^k=[P_m\mu(g_{ij})P_m]_{i=j=1}^k\in B(\cH_m)^{(k)}$ is positive.  That is , for fixed $\x_1,\ldots,\x_k\in\cH$ we have 
\[\sum_{i=j=1}^k \langle P_m\mu(g_{ij})P_m (P_m\x_j),P_m\x_i\rangle=\sum_{i=j=1}^k \langle P_m\mu(g_{ij})P_m \x_j,\x_j\rangle\ge 0.\]
Let $m\to\infty$ and use Lemma \ref{wconvT} to deduce that $\mu_k(G)\ge 0$. Thus $\mu_k$ is positive for each $k\in\N$.
\end{proof}

\noindent
\emph{Proof of Theorem \ref{mainthminf2}.}  Since $\Sigma_1$ and $\Sigma_2$ contain all finite range operators the proof of Theorem \ref{mainthm}  yields that the condition that $L_m(\mu)=[P_{m,2} \mu(\e_{i,1}\e_{j,1}^\vee)P_{m,2}]_{i=j=1}^n\in B(\Sigma_2)^{(n)}$ is positive is necessary.  It is left to show that this condition is sufficient.
Fix $n\in\N$ and let $\cK_n=P_{n,1}\cH_1$.  Let $\mu'_n$ be the restriction of $\mu$ to $B(\cK_n)$.  Assume that $m\ge n$.  The assumption that $L_m(\mu)$ positive yields that $L_{n,m}(\mu)=[P_{m,2} \mu(\e_{i,1}\e_{j,1}^\vee)P_{m,2}]_{i,j=1}^n\in B(\Sigma_2)^{(n)}$ is positive.  Use the prove of part (1) of Theorem \ref{mainthm} to deduce that 
$\mu'_n$ is completely  positive.  Next, fix $m\in\N$ let $\cH_m=P_{m,2}\cH_2$ and  $\hat\mu_m:\Sigma_1\to B(\cH_m)$ be defined as in the proof of part (3) of  Theorem \ref{mainthm}.  As $\mu(P_{n,1}gP_{n,1}), n\in\N$ converges weakly to $\mu(g)$ the proof of part (2) of Theorem  \ref{mainthm} yields that $\hat \mu_m$ is completely positive.
The proof of part (3) of Theorem \ref{mainthm} yields the proof of Theorem \ref{mainthminf2}.\qed
\section{Additional results}\label{sec:addres}
We now recall some standard facts about bounded linear operators on a Hilbert space \cite{RS98}. 
Let $\cH$ be a Hilbert space with the norm $\|\x\|=\sqrt{\langle \x,\x\rangle}$.  For $g\in B(\cH)$ let $\|\g\|=\sup\{\|B\x\|, \|\x\|\le 1\}$ be the operator norm of $g$.
Denote by $T_{\infty}(\cH)\subset B(\cH)$ the closed subspace of compact operators on $\cH$.  Recall that $g\in T_{\infty}(\cH)$ has a convergent singular value decomposition with respect to the operator norm $\|\cdot\|_{\infty}:=\|\cdot\|$:
\begin{eqnarray}\label{SVD}
g=\sum_{i=1}^{\infty}\sigma_i (g)\bk_i\h_i^\vee, \;
\langle \bk_i,\bk_j \rangle=\langle \h_i,\h_j \rangle=\delta_{ij}, i,j\in \N,\\
\|g\|=\sigma_1(g)\ge \cdots\ge\sigma_n(g)\ge\cdots \ge 0,\; \lim_{n\to\infty}\sigma_n(g)=0.\notag
\end{eqnarray}
In particular, $\|g-\sum_{i=1}^n \sigma_i (g)\bk_i\h_i^\vee\|=\sigma_{n+1}(g)$.
Furthermore
\begin{eqnarray*}
g^*=\sum_{i=1}^{\infty}\sigma_i (g)\h_i\bk_i^\vee, \quad\sqrt{gg^*}=\sum_{i=1}^{\infty} \sigma_i (g)\bk_i\bk_i^\vee \ge 0, \quad \sqrt{g^*g}=\sum_{i=1}^{\infty} \sigma_i (g)\h_i\h_i^\vee \ge 0.
\end{eqnarray*}
For $p\in[1,\infty)$ denote by 
\[T_p(\cH)=\{g\in T_\infty(\cH), \|g\|_p:=(\sum_{i=1}^{\infty} |\sigma_i(g)^p)^{\frac{1}{p}}<\infty\}\] the $p$-Schatten class \cite{Sch60,Sim05}.  
In particular for $g\in T_p(\cH)$ we have the equalities 
\begin{eqnarray*}
\|g-\sum_{i=1}^n \sigma_i (g)\bk_i\h_i^\vee\|_p=(\sum_{j=n+1}^\infty\sigma_{j}(g)^p)^{\frac{1}{p}}, \;n\in\N,\\
\|\sqrt{gg^*}\|_p=\|\sqrt{g^*g}\|_p=\|g\|_p=\|g^*\|_p, \quad p\in [1,\infty).
\end{eqnarray*}
Recall that if $a\in B(\cH)$ and $g\in T_p(\cH)$ then $ag,ga\in T_p(\cH)$.  Furthermore
\begin{equation}\label{upbndpnorm}
\max(\|ag\|_p,\|ga\|_p)\le \|a\|\|g\|_p, \quad a\in B(\cH), g\in T_p(\cH), p\in[1,\infty].
\end{equation}
The classes $T_1(\cH)$ and $\cT_2(\cH)$ are called the trace class and the Hilbert-Schmidt class respectively.

The following result is well known and we bring its proof for completeness:
\begin{lemma}\label{normconv}  Let $\cH$ be a countable Hilbert space with an orthonormal basis $\{\e_i\}, i\in\N$.  Let $P_n\in B(\cH)$ be the orthonormal projection on span$(\e_1,\ldots,\e_n)$ for $n\in\N$.  Assume that $p\in[1,\infty]$ and $g\in T_p(\cH)$. Then $\lim_{n\to\infty}\|P_ngP_n-g\|_p=0$.
\end{lemma}
\begin{proof} Assume $g$ has the singular value decomposition \eqref{SVD}.  
Let $g_m=\sum_{i=1}^m \sigma_i (g)\bk_i\h_i^\vee$.
Fix $\varepsilon>0$.  Choose $M(\varepsilon)$ such that $\|g-g_m\|_p<\frac{\varepsilon}{3}$.  Clearly
\begin{eqnarray*}
P_ngP_n-g=(P_ng_mP_n-g_m) +P_n(g-g_m)P_n+(g_m-g).
\end{eqnarray*}
Hence
\begin{eqnarray*}
\|P_ngP_n-g\|_p\le \|P_ng_mP_n-g_m\|_p +\|P_n(g-g_m)P_m\|_p+\|g_m-g\|_p.
\end{eqnarray*}
As $\|P_n\|=1$ it follows that 
\[\|P_n(g-g_m)P_m\|_p\le \|P\| \|g-g_m\|_p\|P_n\|\le\|g-g_m\|_p<\varepsilon/3.\]
Observe next that for $\bu\in\cH$ we have that $\bu^\vee P_n=(P_n\bu)^\vee$.  Lemma \ref{wconvT} yields that 
\[\lim_{n\to\infty}\|P_n\bk_i-\bk_i\|=\lim_{n\to\infty}\|(P_n\h_i)^\vee-\h_i^\vee\|=0 \textrm{ for } i\in\N.\]
Hence $\lim_{n\to\infty}\|P_n(\bk_i\h_i^\vee)P_n-\bk_i\h_i^\vee\|=0$ for each $i\in\N$.
In particular $\lim_{n\to\infty}\|P_ng_mP_n-g_m\|=0$.  Therefore $\lim_{n\to\infty} \sigma_i(P_n g_m P_n-g_m)=0$ for each $i\in\N$.
As the dimension of the range of $g_m$ and $P_ng_mP_n$ is at most $m$ it follows that $\sigma_i(P_ng_mP_n-g_m)=0$ for $i>2m$.  Hence
there exists $N(\varepsilon)\in\N$ such that $\|P_ng_mP_n-g_m\|_p<\varepsilon/3$ for $n>N(\varepsilon)$.  In particular, $\|P_n g P_n -g\|_p<\varepsilon$.
\end{proof}
\begin{corollary}\label{Pnmuconv} Let $\cK$ be an infinite dimensional separable Hilbert space, and $\Sigma_2$ be a Banach space.  Let $P_n:\cK\to\cK$ be the orthogonal projection on span$(\bk_1,\ldots,\bk_n)$ for $n\in\N$.  Assume that $\mu: T_p(\cK)\to \Sigma_2,p\in[1,\infty]$ is a bounded linear operator.  Then for each $g\in T_p(\cK)$ $\lim_{n\to\infty}\|\mu(P_ngP_n)-\mu(g)\|=0$.
\end{corollary}

Assume that in Theorem \ref{mainthminf2} $\Sigma_1=T_p(\cH_1)$ for some $p\in[1,\infty]$.  Then the assumption (4) holds.  

We now give a well known condition which ensures the condition (4) on $\mu$ in Theorem \ref{mainthminf2} is satisfied.  Let $S(\cK)\subset B(\cK)$ be the real subspace of selfadjoint operators on the Hilbert space $\cK$. Let $K(B(\cK))\subset S(\cK)$ be the cone of positive operators.
Then $K(B(\cK))$ is a closed, pointed ($K(B(\cK))\cap(-K(B(\cK)))=\{0\}$, generating ($S(\cK)=K(B(\cK))-K(B(\cK))$) cone with interior.  For $a,b\in S(\cK)$ denote by $a\ge b$ if $a-b\ge 0$ ($a-b\in K(B(\cK))$).  Let $K(T_p(\cK)):=K(B(\cK))\cap T_p(\cK)$ for $p\in[1,\infty]$ be a closed pointed generating cone in $S(B(\cK))\cap T_p(\cK)$.  Recall that $B(\cK)$ is the dual space of $T_1(\cK)$, where $a(g)= \tr(ga)$ for $a\in B(\cK)$ and $g\in T_1(\cK)$ \cite[Theorem VI.26 (b)]{RS98} or \cite{Sim05}.
The weak star topology ($w^*$) on $B(\cK)$ is the weak operator topology, which is also called the ultraweak topology for $C^*$ algebras.
Recall that $K(B(\cK))$ is the dual cone of $K(T_1(\cK))$.  A sequence $a_n\in K(B(\cH))$ is called increasing if $0\le a_1\le \cdots\le a_n\le \cdots$.  
Assume that above sequence is bounded by norm.  Then there exists a subnet of this sequence that converge in $w^*$ topology to $a\in K(B(\cK))$ \cite{Kar59}.  This is equivalent to the statement that $\lim_{n\to\infty}=\langle a_n\x,\x\rangle = \langle a\x,\x\rangle$ for each $\x\in\cK$.
We denote $a$ $\sup a_n$.   Assume that $\mu: B(\cK)\to B(\cH)$ is a bounded positive  linear map ($\mu(K(B(\cK)))\subset K(B(\cH))$.  Then  $\mu$ is called  is called normal  if for each increasing seqeunce $a_n\in K(B(\cK))$ we have the equality
$\mu(\sup a_n)=\sup \mu(a_n)$.  This is equivalent to the assumptions that every $w^*$ convergent net in $B(\cK)$ is mapped to a $w^*$ convergent net in $B(\cH)$ \cite{Kar59,Sak}.  Thus if in Theorem \ref{mainthminf2} $\mu$ is normal the condition $\mu$ maps strongly convergent sequence to weakly convergent sequences are satisfied.
\section{Representation of quantum channels}\label{sec:repqc}
Denote by $S(B(\C^n))=\rH_n\supset \rH_{n,+}=K(B(\C^n))\supset \rH_{n,+,1}$ the real space of $n\times n$ hermitian matrices, the cone of positive (semidefinite) matrices and the convex set of density matrices.
Let $\mu:\C^{n\times n}\to \C^{m\times m}$ be a completely positive map:
\begin{equation}\label{fdimcp}
\mu(X)=\sum_{i=1}^k A_i X A_i^*, \quad A_i\in\C^{m\times n}, i\in[k].
\end{equation}
A standard definition of a quantum channel is a completely positive map $\mu$  that preserves  density matrices.  This is equivalent to the assumption that $\mu$ is trace preserving, i.e. $\sum_{i=1}^k A_i^* A_i = I_n$ \cite{FL16}.  A more general definiton can be stated as in \cite{HK69,Kra71}.  We call a nonzero completely positive operator $\mu$ given by \eqref{fdimcp}  a quantum subchannel if $\tr X\ge \tr \mu(X)$ for each  $X\in\rH_{n,+,1}$.  This is equivalent to the assumption, e.g. \cite{FL16},
\begin{equation}\label{subchancond}
0\lneq \sum_{i=1}^k A_i^*A_i\le I_n.
\end{equation}
Then the nonlinear transformation $X\mapsto \frac{1}{\tr X}\mu(X)$ is a nonlinear quantum channel, which takes $X\in\rH_{n,+,1}$ to $\frac{1}{\tr X}\mu(X)$ with probability  $\tr \mu(X)$.

Let $\cK$ be a separable infinite dimensional Hilbert space.  Then $g\in T_1(\cK)$ is called a density operator (matrix) if $g\ge 0$ and $\tr g=1$.  Let $\cH$ be another separable infinite dimensional Hilbert space.  
Then $\mu: T_1(\cK)\to T_1(\cH)$ is called quantum channel if $\mu$ is completely positive operator which preserves the density operators.  This is equivalent to the assumption that $\mu$ is trace preserving.
We call $\mu$ a quantum subchannel if $\mu$ is completely positive operator such that for each density operator $g$ we have $\tr g\ge \tr \mu(g)$.  The following theorem gives a simple neccessary and sufficient conditions for $\mu: T_1(\cK)\to T_1(\cH)$ to be a quantum subchannel.  Different necessary and sufficient conditions for trace preserving map $\mu: T_1(\cK)\to T_1(\cH)$ are given \cite{Hol11}.
\begin{theorem}\label{charsubchan}  Let $\cK$ and $\cH$ be Hilbert spaces with infinite countable bases $\{\bk_i\},\{\h_i\}, i\in\N,$ respectively.  For each $n\in\N$ denote by $P_n$ the orthogonal projection of $\cH$ on span$(\h_1,\ldots,\h_n)$ Assume that  $\mu: T_1(\cK)\to T_1(\cH)$ is a bounded linear operator.  Then  $\mu$ is a subchannel if and only for each $n\in\N$ if the following conditions hold:
\begin{enumerate}
\item The matrix $L_n(\mu)=[P_n\mu(\bk_i\bk_j^\vee) P_n]_{i,j\in[n]}\in B(\cH)^{(n)}$ is positive.

\item The $n\times n$ matrix $M_n:=[\tr P_n\mu(\bk_i\bk_i^\vee)P_n]_{i,j\in[n]}$ satisfies $M_n\le I_n$.
\end{enumerate}
Furthermore, $\mu$ is a trace preserving channel if and only if the following conditions  holds: First the above condition (1) holds for each $n\in\N$.  Second 
\begin{equation}\label{tracepreser}
\tr \mu(\bk_i \bk_j^\vee)=\delta_{ij}, \quad i,j\in\N.
\end{equation} 
\end{theorem}
\begin{proof}  In view of Theorem \ref{mainthminf2} and Lemma \ref{normconv} it follows that the condition (1) is a necessary and sufficient conditions for $\mu$ to be completely positive.  We now show that if $\mu$ is a quantum subchannel then the condition (2) holds.  Suppose that $g\in T_1(\cH)$.  Recall that  $\tr g=\sum_{i=1}^{\infty} \langle g\h_i,\h_i\rangle$.  Assume that $g\ge 0$.  Then
\[\tr g\ge \sum_{i=1}^n \langle g\h_i,\h_i\rangle=\tr P_n g P_n.\]
Assume that $f\in T_1(\cK)$ and $f\ge 0$.  The assumption that $\mu$ is a subchannel yields that $\tr f\ge \tr \mu(f)\ge \tr P_n\mu(f)P_n$.  Choose $f=\bu\bu^\vee$, where $\bu=\sum_{i=1} x_i \bk_i$.  Let $\x=(x_1,\ldots,x_n)\trans$.  Thus $\tr f=\x^*\x$ and
$\tr P_n\mu(f)P_n=\x^*M_n\x$.  Hence $M_n\le I_n$.

It is left to show that if the condition (2) is sufficient for a completely positive map to be a quantum subchannel.  Fix $n\in\N$ and let $m\ge n$.  Let $f=\bu\bu^\vee$ be defined as above.  We then obtain that $\tr f\ge \tr P_m\mu(f)P_m$.  Use the arguments of the proof of Theorem \ref{mainthminf2} to deduce that $\tr f\ge \tr \mu(f)$.  Assume that 
$a\in T_1(\cK), a\ge 0$ and the range of $a$ is contained in span$(\bk_1,\ldots,\bk_n)$.  Then $a=\sum_{j=1}^n \bu_j\bu_j^\vee$ f or some $\bu_1,\ldots,\bu_n\in$span$(\bk_1,\ldots,\bk_n)$.  Hence 
\[\tr a=\sum_{j=1}^n\tr \bu_j\bu_j^\vee\ge \sum_{j=1}^n\tr \mu(\bu_j\bu_j^\vee)=\tr \mu(a).\] 
Let $b\in T_1(\cK), b\ge 0$.  Then $P_nbP_n$ is of the above form and we get that
$\tr P_n bP_n\ge \tr \mu(P_n b P_n)$.  Let $n\to\infty$ and use Lemma \ref{normconv}
to deduce that $\tr b \ge \tr \mu(b)$.

It is left to discuss the necessary and sufficient conditions for a completely positive $\mu$ to be trace preserving.  Suppose first that $\mu$ is trace preserving.  Clearly, condition (1) holds.  As $\mu$ is trace preserving we deduce that the condition \eqref{tracepreser} holds.  It is left to show that the condition (1) and  \eqref{tracepreser} yield that $\mu$ is quantum channel.  As in the previous case the condition (1) ensures that $\mu$ is completely positive.  It is left to show that $\mu$ is trace preserving.  Assume as above that $a\ge 0$ and range of $a$ is contained in span$(\bk_1,\ldots,\bk_n)$.  The the condition  \eqref{tracepreser} yields that $\tr a=\tr \mu(a)$.  Use the above arguments for $b\ge 0$ to deduce that $\tr b=\tr \mu(b)$.
As every selfadjoint $b\in T_1(\cK)$ is of the form $b_+-b_-$, where $b_+,b_-\in T_-(\cK)$ and $b_+,b_-\ge 0$ it follows that $\tr b=\tr \mu(b)$ for each selfadjoint in $T_1(\cK)$.  As $\mu$ is linear it follows that $\mu$ is trace preserving on $T_1(\cK)$. 

\end{proof} 

Recall that the dual space of $T_1(\cH)$, denoted as $T_1(\cH)^\vee$,  is $B(\cH)$.
Assume that $\mu: T_1(\cK)\to T_1(\cH)$ be a bounded linear operator.  Denote by $\mu^\vee$ the induced linear operator $\mu^\vee: B(\cH)\to B(\cK)$.
\begin{theorem}\label{composmus}   Let $\cK$ and $\cH$ be Hilbert spaces with infinite countable bases $\{\bk_i\},\{\h_i\}, i\in\N,$ respectively.  Assume that the bounded linear operator $\mu: T_1(\cK)\to T_1(\cH)$ is completely positive.  Then the map $\mu^\vee: B(\cH)\to B(\cK)$ is completely positive.
\end{theorem}
\begin{proof}  As $\mu$ is completely positive it follows that $\mu$ maps seladjoint operators to selfadjoint.  Hence $\mu(a^*)=\mu(a)^*$.  Therefore $\mu^\vee(b^*)=(\mu^\vee(b))^*$.  We now apply Theorem \ref{mainthminf2}.  
Thus $\cH_1=\cH, \cH_2=\cK, \Sigma_1=B(\cH), \Sigma_2=B(\cK)$.   

We first observe that condition (4) holds.  Let $P_n$ be the orthogonal projection of $\cH$ on span$(\h_1,\ldots,\h_n)$.  Assume that $b\in B(\cH)$.  We need to show that  $\lim_{n\to\infty} \langle \mu^\vee(P_nbP_n)\x,\y\rangle=\langle \mu^\vee(b)\x,\y\rangle$ for each pair $\x,\y\in \cK$. Clearly
\[\langle \mu^\vee(P_nbP_n)\x,\y\rangle=\tr  \mu^\vee(P_nbP_n)\x\y^\vee=\tr P_nbP_n\mu(\x\y^\vee).\]
Write down the singular value decomposition of $\mu(\x\y^\vee)\in T_1(\cH)$:
\[ \mu(\x\y^\vee)=\sum_{i=1}^{\infty}\sigma_i \bu_i\bv_i^\vee, \quad \sum_{i=1}^\infty \sigma_i<\infty, \langle \bu_i,\bu_j\rangle=\langle \bv_i,\bv_j\rangle=\delta_{ij}, i,j\in\N.\]
Then
\[\tr P_nbP_n \bu_i\bv_i^\vee=\langle P_nbP_n \bu_i,\bv_i\rangle,\;|\tr P_nbP_n \bu_i\bv_i^\vee|\le \|b\|, \quad \lim_{n\to\infty}\tr P_nbP_n \bu_i\bv_i^\vee=\tr  b \bu_i\bv_i^\vee, \quad i\in\N.\]
Therefore 
\[\lim_{n\to\infty} \langle \mu^\vee(P_nbP_n)\x,\y\rangle=\sum_{i=1}^{\infty} \sigma_i(\tr b\bu_i\bv_i^\vee)=\tr b \mu(\x\y^\vee)=\langle \mu^\vee(b)\x,\y\rangle.\]

Let $Q_n:\cK\to\cK$ be the orthogonal projection on span$(\bk_1,\ldots,\bk_n)$.   
It is left to show that for each $n\in\N$ the matrix $L_n(\mu^\vee)=[Q_n\mu^\vee(\h_i\h_j^\vee)Q_n]_{i,j\in[n]}\in B(\cK)^{(n)}$ is positive.  That is, 
\[0\le \sum_{i,j=1}^n \langle\mu^\vee(\h_i\h_j^\vee)\x_i,\x_j\rangle=\sum_{i,j=1}^n \tr\mu^\vee(\h_i\h_j^\vee)\x_i\x_j^\vee=\sum_{i,j=1}^n \tr\h_i\h_j^\vee\mu(\x_i\x_j^\vee)=\sum_{i,j=1}^n \langle\mu(\x_i\x_j^\vee)\h_i,\h_j\rangle\]
for a given $\x_1,\ldots,\x_n\in$span$(\bk_1,\ldots,\bk_n)$.  The above inequality follows from the assumption that the map $\mu: T_1(\cK)\to T_1(\cH)$ is completely positive.

\end{proof}

Following the papers of Hellwig-Kraus \cite{HK69} and Kraus \cite{Kra71} we explain how to generate special quantum subchannel $\mu: T_1(\cK)\to T_1(\cH)$ using the ``trace out'' notion.  We consider two quantum systems which are represented as density operators $a\in T_1(\cK)$ and  $b\in T_1(\cH)$.  The combined system is described by $a\otimes b\in T_1(\cK\otimes\cH)$.  The combined system evolves to a new system $c=U(a\otimes b)U^*$, where $U$ is a unitary operator on $\cK\otimes\cH$.  
 Let $Q:\cK\to \cK$ be an orthogonal projection.  Denote $\hat{\cK}=Q\cK$. 
 Assume that  $\hat\bk_i, i\in[N]$ is an orthonormal basis of $\hat\cK$.  (Here either $N\in\N$ or $N=\infty$ and $[\infty]=\N$.)  
  Let $e=(Q\otimes Id) d(Q\otimes Id)\in T_1(\cK_1\otimes \cH)$.  Note that $e\ge 0$ and $\tr e\le 1$.  Denote by 
 \[\tr_{\hat\cK} (e)=\sum_{i=j=1}^\infty(\sum_{p\in[N]} \langle e (\hat\bk_p\otimes \h_i), \hat\bk_p\otimes \h_j\rangle)\h_i\h_j^\vee.\]
 This is the ``trace out'' operation.  It is well known that  $\tr_{\hat\cK} (e)\in T_1(\cH)$, $e\ge 0$, and $\tr (\tr_{\hat\cK} (e))=\tr e$. 
Then we obtain a new system $\mu(a)=\tr_{\hat\cK} (e)\in T_1(\cH)$.  As in \cite{HK69} it follows that
\begin{equation}\label{HKqsc}
\mu(a)=\sum_{i=1}^{\infty} A_ia A_i^\vee, \; A_i\in L(\cK,\cH),\;i\in N,\quad \sum_{i=1}^{\infty} A_i^\vee A_i\le Id.
\end{equation}
For $\hat\cK=\cK$ we get the equality $\sum_{i=1}^{\infty} A_i^\vee A_i=Id$.  

 Note that the convergence of the infinite sum in \eqref{HKqsc} is in the weak operator topology.  Furthermore, it is straightforward to show that if $\mu$ is a quantum subchannel given by \eqref{HKqsc} then the dual $\mu^\vee: B(\cK)\to B(\cH)$ is given by $\mu^\vee(c)=\sum_{i=1}^\infty A_i^\vee c A_i$.
 
 It is shown in \cite{Hol11} that every trace preserving $\mu: T_1(\cK)\to T_1(\cH)$, where $\cK$ and $\cH$ are separable, has the Krauss representation \eqref{HKqsc} with $\sum_{i=1}^\infty A_i^\vee A_i=Id$.  We now use the arguments in \cite{Kra71,Hol11}
 to show that any quantum subchannel $\mu:\cK\to\cH$ has the form \eqref{HKqsc}. 
 For $\cK=\cH$ this result is \cite[Theorem 3.3]{Kra71} under the additional assumption that $\mu$ is continuous in the weak operator topology.   
 \begin{theorem}\label{subchthm}  Let $\cK$ and $\cH$ be two separable Hilbert spaces.  Then $\mu:\cK\to\cH$ is a quantum subchannel if and only if there exist bounded linear operators $A_i\in L(B(\cK),B(\cH))$ for $i\in\N$ such that \eqref{HKqsc} holds,
 \end{theorem}
 \begin{proof}  Assume that $\mu: \cK\to \cH$ is a quantum subchannel.
 Let $\cK\circ \cH^\vee$ be the algebraic tensor product of $\cK$ and $\cH^\vee$, (the dual space of $\cH$).  That is, the elements of $\cK\circ \cH^\vee$  are finite sums of the form $\sum_{i=1}^N \x_i\otimes \bu_i^\vee$ for $\x_i\in\cK,\bu_i\in \cH,N\in\N$.  We define the following sesquilinear form on $\cK\circ \cH$:
 \[\Omega_{\mu}(\sum_{i=p=1}^N \x_i\otimes\bu_p^\vee,\sum_{j=q=1}^N\y_j\otimes\bv_q^\vee)=\sum_{i=j=p=q=1}^N \langle\mu(\x_i\y_j^\vee)\bv_q,\bu_p\rangle.\]
 The assumption that $\mu$ is completely positive means that $\Omega_{\mu}$ is positive semidefinite.  Then $\Omega_{\mu}$ defines a semiinner product and a seminorm 
 \[\langle \z,\bw\rangle_{\mu}=\Omega_\mu(\z,\bw),\quad \| \z\|_{\mu}=\sqrt{\Omega_\mu(\z,\z)}, \quad \z,\bw\in \cK\circ\cH^\vee.\]
 Clearly, we have the Cauchy-Schwarz inequality $|\langle \z,\bw\rangle_{\mu}|\le \| \z\|_{\mu}\| \bw\|_{\mu}$.  Let $\cL$ be the closure of $\cK\circ\cH^\vee$ with respect to the semimetric $\|\cdot\|_\mu$.  Then we can extend the semiinner product $\langle \cdot,\cdot\rangle_\mu$ and the seminorm $\|\cdot\|_\mu$ to $\cL$.
 Denote by $\cM$ the linear subspace of all $\z\in\cL$ such that $\|\z\|_\mu=0$.  Then $\cN=\cL/\cM$ is a separable Hilbert space.  Assume for simplicity of the exposition that $\cN$ is infinite dimensional. Let $\f_l,l\in\N$ be an orthonormal basis of $\cN$.  For $l\in\N$ denote by $\phi_l:\cN\to \C$ the linear functional $\phi_l(\z)=\langle \z,\f_l\rangle_\mu$.  Then $\phi_l$ extends to a linear functional $\hat\phi_l:\cM\to\C$: $\hat\phi_l(\bw)=\langle \bw,\f_l\rangle_\mu$ and $|\langle \bw,\f_l\rangle_\mu|\le \|\bw\|_\mu$.  Moreover
 \[\bw=\sum_{l=1}^\infty \langle \bw,\f_l\rangle_\mu \f_l, \quad \|\bw\|_\mu^2=\sum_{l=1}^\infty |\langle \bw,\f_l\rangle_\mu|^2,\;\bw\in\cM.\]
 In particular
 \[\|\x\otimes \bu^\vee\|^2_\mu=\langle \mu(\x\x^\vee)\bu,\bu\rangle=\sum_{l=1}^\infty |\langle \x\otimes\bu^\vee,\f_l\rangle_\mu|^2.\]
 Assume that $\bu\ne 0$ and let $\bu_i,i\in\N$ be an orthonormal basis, where $\bu_1=\frac{1}{\|\bu\|}\bu$.  Then $\tr \mu(\x\x^\vee)=\sum_{i=1}^\infty \langle \mu(\x\x^\vee)\bu_i,\bu_i\rangle$.  As $\mu$ is subchannel we obtain
 \begin{equation}\label{subchanin}
 \sum_{i=1}^\infty \langle \mu(\x\x^\vee)\bu_i,\bu_i\rangle \le \tr \x\x^\vee=\|\x\|^2.
 \end{equation}
 In particular $\|\x\otimes\bu^\vee\|_\mu\le \|\x\|\|\bu\|$.
 Define $A_l:\cK\to \cH$ by equality $\langle A_l\x,\bu\rangle=\langle \x\otimes \bu^\vee,\f_l\rangle_\mu$.  So $\|A_l\|\le 1$ and $A_l\x=\sum_{i=1}^\infty \langle \x\otimes \bu_i^\vee,\f_l\rangle_\mu \bu_i$. Thus
 \begin{eqnarray*}
 \langle(\sum_{l=1}^N A_l^\vee A_l)\x,\x\rangle=\sum_{l=1}^N \langle A_l^\vee A_l\x,\x\rangle=\sum_{l=1}^N \langle  A_l\x,A_l\x\rangle=\sum_{l=1}^N \langle\sum_{i=1}^\infty  \langle \x\otimes \bu_i^\vee,\f_l\rangle_\mu \bu_i,\sum_{j=1}^\infty  \langle \x\otimes \bu_j^\vee\rangle_\mu\bu_j\rangle=\\
 \sum_{l=1}^N \sum_{i=1}^\infty  |\langle \x\otimes \bu_i^\vee,\f_l\rangle_\mu|^2=
 \sum_{i=1}^\infty \sum_{l=1}^N   |\langle \x\otimes \bu_i^\vee,\f_l\rangle_\mu|^2\le 
  \sum_{i=1}^\infty \langle \mu(\x\x^\vee)\bu_i,\bu_i\rangle \le \tr \x\x^\vee=\|\x\|^2.
 \end{eqnarray*}
 Thus $\sum_{l=1}^N A_l^\vee A_l\le Id$ for each $N\in\N$.  Hence $\sum_{l=1}^\infty A_l^\vee A_l\le I_d$.  
 
 It is left to show that $\mu$ is given by \eqref{HKqsc}.  
 Let $\x,\y\in\cK$.  Then $\x\y^\vee\in T_1(\cK)$.  We need to show that $\mu(\x\y^\vee)=\sum_{l=1}^\infty (A_l\x) (A_l\y)^\vee$.  Thus, it is enough to show that 
 \[\langle \mu(\x\y^\vee)\bu,\bv\rangle=\sum_{l=1}^\infty \langle \bu, A_l\y\rangle \langle A_l\x,\bv\rangle \quad \textrm{for all } \x,\y\in\cK, \bu,\bv\in\cH.\]
 Recall that $\langle A_l\x,\bv\rangle=\langle \x\otimes\bv^\vee, \f_l\rangle_\mu$ and $\langle \bu, A_l\y\rangle=\overline{\langle A_l\y, \bu\rangle}=\overline{\langle \y\otimes\bu^\vee, \f_l\rangle_\mu}$.  Hence
 \[\sum_{l=1}^\infty \langle \bu, A_l\y\rangle \langle A_l\x,\bv\rangle=\sum_{l=1}^\infty 
 \langle \x\otimes\bv^\vee, \f_l\rangle_\mu\overline{\langle \y\otimes\bu^\vee, \f_l\rangle_\mu}=\langle  \x\otimes\bv^\vee,\y\otimes\bu^\vee\rangle_\mu=\langle \mu(\x\y^\vee)\bu,\bv\rangle.\]
 
  Assume that \eqref{HKqsc} holds.  For $N\in\N$ define $\mu_N:B(\cK)\to B(\cH)$ by the equality $\mu_N(a)=\sum_{l=1}^N A_l a A_l^\vee$. Clearly, $\mu_N$ is completely positive.  Also $\mu_N$ is a quantum subchannel as $\mu_N: T_1(\cK)\to T_1(\cH)$.
 Therefore $\mu$ given by \eqref{HKqsc} is a subchannel.
 \end{proof}
 
 \noindent
 \emph{Acknowledgement}.  Part of this work was carried out when the author visited 
Institute of Systems Science, 
Academy of Mathematics and System Sciences,
Academia Sinica, Beijing, China, in May 2018.  I thank my host Professor Lihong Zhi for raising the problem of generalizing Choi's theorem to the infinite dimensional case.

 \end{document}